\documentclass[12pt]{amsart}%
\usepackage{graphicx,amscd,color,amsmath,amsfonts,amssymb,geometry}
\usepackage{amsmath}
\usepackage{amsfonts}
\usepackage{amssymb}
\usepackage{graphicx}%
\providecommand{\U}[1]{\protect\rule{.1in}{.1in}}
\providecommand{\U}[1]{\protect\rule{.1in}{.1in}}
\providecommand{\U}[1]{\protect\rule{.1in}{.1in}}
\providecommand{\U}[1]{\protect\rule{.1in}{.1in}} \textwidth 17cm
\theoremstyle{plain}

\newtheorem{theorem}{Theorem}[section]
\newtheorem{proposition}[theorem]{Proposition}

\newtheorem{remark}[theorem]{Remark}

\newtheorem{lemma}[theorem]{Lemma}
\newtheorem{problem}[theorem]{Problem}
\newtheorem{definition}[theorem]{Definition}
\numberwithin{equation}{section}

\begin{document}
\title[Some techniques on nonlinear analysis and applications]{Some techniques on nonlinear analysis and applications}
\author{Daniel Pellegrino}
\address{Departamento de Matem\'{a}tica, Universidade Federal da Para\'{\i}ba,
58.051-900 - Jo\~{a}o Pessoa, Brazil.}
\email{dmpellegrino@gmail.com}
\author{Joedson Santos}
\address{Departamento de Matem\'{a}tica, Universidade Federal de Sergipe, 49.500-000 -
Itabaiana, Brazil.}
\email{joedsonsr@yahoo.com.br}
\author{Juan B. Seoane-Sep\'{u}lveda}
\address{Facultad de Ciencias Matem\'{a}ticas. Departamento de An\'{a}lisis
Matem\'{a}tico. Universidad Complutense de Madrid. Plaza de las Ciencias 3.
28040. Madrid, Spain.}
\email{jseoane@mat.ucm.es}
\thanks{Daniel Pellegrino was supported by CNPq Grant 301237/2009-3. }
\thanks{Juan B. Seoane-Sep\'{u}lveda was supported by the Spanish Ministry of Science
and Innovation, grant MTM2009-07848.}

\begin{abstract}
In this paper we present two different results in the context of nonlinear
analysis. The first one is essentially a nonlinear technique that, in view of
its strong generality, may be useful in different practical problems. The
second result, more technical, but also connected to the first one, is an
extension of the well-known Pietsch Domination Theorem. The last decade
witnessed the birth of different families of Pietsch Domination-type results
and some attempts of unification. Our result, that we call ``full general
Pietsch Domination Theorem'' is potentially a definitive Pietsch Domination
Theorem which unifies the previous versions and delimits what can be proved in
this line.
%The proof uses precisely Pietsch's original insight, now in a nonlinear disguise (our technique is an
%abstraction of a recent approach used by the first and second named authors in
%\cite{psmz}).
The connections to the recent notion of weighted summability are traced.

\end{abstract}
\maketitle

\section{Introduction and motivation}

\label{intro}Common, even simple, mathematical problems usually involve
nonlinear maps, sometimes acting on sets with little (or none) algebraic
structure; so the extension of linear techniques to the nonlinear setting,
besides its intrinsic mathematical interest, is an important task for
potential applications. In fact it is mostly a challenging task, since linear
arguments are commonly ineffective in a more general setting. The following
problem illustrates this situation.

If $X,Y$ are Banach spaces, $u,v:X\rightarrow Y$ are continuous linear
operators, $C>0$ and $1\leq p\leq q<\infty$ it is possible to show that:

\begin{enumerate}
\item[\textbf{1.-)}] If
\begin{equation}
\text{ }%
%TCIMACRO{\tsum \limits_{j=1}^{m}}%
%BeginExpansion
{\textstyle\sum\limits_{j=1}^{m}}
%EndExpansion
\left\Vert u(x_{j})\right\Vert ^{p}\leq C%
%TCIMACRO{\tsum \limits_{j=1}^{m}}%
%BeginExpansion
{\textstyle\sum\limits_{j=1}^{m}}
%EndExpansion
\left\Vert v(x_{j})\right\Vert ^{p}\text{ for every }m\text{ and all }%
x_{1},...,x_{m}\in X,\text{ } \label{ree}%
\end{equation}
then%
\begin{equation}
\text{ }%
%TCIMACRO{\tsum \limits_{j=1}^{m}}%
%BeginExpansion
{\textstyle\sum\limits_{j=1}^{m}}
%EndExpansion
\left\Vert u(x_{j})\right\Vert ^{q}\leq C%
%TCIMACRO{\tsum \limits_{j=1}^{m}}%
%BeginExpansion
{\textstyle\sum\limits_{j=1}^{m}}
%EndExpansion
\left\Vert v(x_{j})\right\Vert ^{q}\text{ for every }m\text{ and all }%
x_{1},...,x_{m}\in X. \label{ree2}%
\end{equation}
Also, in the same direction:

\item[\textbf{2.-)}] If%
\begin{equation}%
%TCIMACRO{\tsum \limits_{j=1}^{m}}%
%BeginExpansion
{\textstyle\sum\limits_{j=1}^{m}}
%EndExpansion
\left\Vert u(x_{j})\right\Vert ^{p}\leq C\sup_{\varphi\in B_{X^{\ast}}}%
%TCIMACRO{\tsum \limits_{j=1}^{m}}%
%BeginExpansion
{\textstyle\sum\limits_{j=1}^{m}}
%EndExpansion
\left\vert \varphi(x_{j})\right\vert ^{p}\text{ for every }m\text{ and all
}x_{1},...,x_{m}\in X, \label{wqqq}%
\end{equation}
then%
\begin{equation}%
%TCIMACRO{\tsum \limits_{j=1}^{m}}%
%BeginExpansion
{\textstyle\sum\limits_{j=1}^{m}}
%EndExpansion
\left\Vert u(x_{j})\right\Vert ^{q}\leq C\sup_{\varphi\in B_{X^{\ast}}}%
%TCIMACRO{\tsum \limits_{j=1}^{m}}%
%BeginExpansion
{\textstyle\sum\limits_{j=1}^{m}}
%EndExpansion
\left\vert \varphi(x_{j})\right\vert ^{q}\text{ for every }m\text{ and all
}x_{1},...,x_{m}\in X, \label{wqqq2}%
\end{equation}
where $X^{\ast}$ is the topological dual of $X$ and $B_{X^{\ast}}$ denotes its
closed unit ball.\ More generally, if
\begin{equation}%
\begin{array}
[c]{c}%
p_{j}\leq q_{j}\text{ for }j=1,2,\\
1\leq p_{1}\leq p_{2}<\infty,\\
1\leq q_{1}\leq q_{2}<\infty,\\
\frac{1}{p_{1}}-\frac{1}{q_{1}}\leq\frac{1}{p_{2}}-\frac{1}{q_{2}},
\end{array}
\label{wer}%
\end{equation}
then
\begin{equation}
\left(
%TCIMACRO{\tsum \limits_{j=1}^{m}}%
%BeginExpansion
{\textstyle\sum\limits_{j=1}^{m}}
%EndExpansion
\left\Vert u(x_{j})\right\Vert ^{q_{1}}\right)  ^{1/q_{1}}\leq C\sup
_{\varphi\in B_{X^{\ast}}}\left(
%TCIMACRO{\tsum \limits_{j=1}^{m}}%
%BeginExpansion
{\textstyle\sum\limits_{j=1}^{m}}
%EndExpansion
\left\vert \varphi(x_{j})\right\vert ^{p_{1}}\right)  ^{1/p_{1}}\text{ for
every }m\text{ and all }x_{1},...,x_{m}\in X \label{d1}%
\end{equation}
implies that%
\begin{equation}
\left(
%TCIMACRO{\tsum \limits_{j=1}^{m}}%
%BeginExpansion
{\textstyle\sum\limits_{j=1}^{m}}
%EndExpansion
\left\Vert u(x_{j})\right\Vert ^{q_{2}}\right)  ^{1/q_{2}}\leq C\sup
_{\varphi\in B_{X^{\ast}}}\left(
%TCIMACRO{\tsum \limits_{j=1}^{m}}%
%BeginExpansion
{\textstyle\sum\limits_{j=1}^{m}}
%EndExpansion
\left\vert \varphi(x_{j})\right\vert ^{p_{2}}\right)  ^{1/p_{2}}\text{ for
every }m\text{ and all }x_{1},...,x_{m}\in X. \label{d2}%
\end{equation}

\end{enumerate}

\begin{problem}
What about nonlinear versions of the above results? Are there any?
\end{problem}

\begin{problem}
What about nonlinear versions in which the spaces $X$ and $Y$ are just sets,
with no structure at all?
\end{problem}

The interested reader can find the proof of the implication (\ref{d1}%
)$\Rightarrow$(\ref{d2}) in \cite[p. 198]{DJT}. This result was essentially
proved by S. Kwapie\'{n} in 1968 (see \cite{kw68}) and it is what is now
called ``Inclusion Theorem for absolutely summing operators''. A quick look
shows that the linearity is fully explored and a nonlinear version of this
result, if there is any, would require a whole new technique. It is worth
mentioning that practical problems may also involve sets with less structure
than Banach spaces (or less structure than linear spaces or even than metric
spaces) and a ``full'' nonlinear version (with no structure on the spaces
involved) would certainly be interesting for potential applications.

In this direction we will prove a very general result, which we will call
``Inclusion Principle'', which, due its extreme generality, may be useful in
different contexts, even outside of pure mathematical analysis. The arguments
used in the proof of the ``Inclusion Principle'' are, albeit tricky, fairly
clear and simple in nature, but we do believe this technique may be useful in
different contexts. To illustrate its reach, at least in the context of
Functional Analysis, we show that very particular cases of the Inclusion
Principle can contribute to the nonlinear theory of absolutely summing operators.

Below, as an illustration, we describe an extremely particular case of the
forthcoming Inclusion Principle:

Let $X$ be an arbitrary non-void set and $Y$ be a normed space; suppose that
$p_{j}$ and $q_{j}$ satisfy (\ref{wer}). If $f,g:X\rightarrow Y$ are arbitrary
mappings and there is a constant $C>0$ so that%
\[
\sum_{j=1}^{m}\left\Vert f(x_{j})\right\Vert ^{q_{1}}\leq C\sum_{j=1}%
^{m}\left\Vert g(x_{j})\right\Vert ^{p_{1}},
\]
for every $m$ and all $x_{1},\ldots,x_{m}\in X$, then there is a constant
$C_{1}>0$ such that
\[
\left(  \sum_{j=1}^{m}\left\Vert f(x_{j})\right\Vert ^{q_{2}}\right)
^{\frac{1}{\alpha}}\leq C_{1}\sum_{j=1}^{m}\left\Vert g(x_{j})\right\Vert
^{p_{2}}%
\]
for every $m$ and all $x_{1},\ldots,x_{m}\in X$, with%
\[
\alpha=\frac{q_{2}p_{1}}{q_{1}p_{2}}\text{ if }p_{1}<p_{2}.
\]
The case $p_{1}=p_{2}$ is trivial$.$ The parameter $\alpha$ is a kind of
adjustment, i.e., the price that one has to pay for the complete lack of
linearity, and precisely when $p_{j}=q_{j}$ for $j=1$ and $2$ we have
$\alpha=1$ and no adjustment is needed. In other words, the parameter $\alpha$
indicates the necessary adjustments (in view of the lack of linearity) when
$p_{j}$ and $q_{j}$ become distant.

The second main contribution of this paper is more technical, but also useful.
It is what we call ``full general Pietsch Domination Theorem'' which, as will
be shown, has several applications and seems to be a definitive answer to the
attempt of delimiting the amplitude of action of Pietsch Domination-type theorems.

The Pietsch Domination Theorem (PDT) (sometimes stated as the Pietsch
Factorization Theorem) was proved in 1967 by A. Pietsch, in his classical
paper \cite{stu}, and since then it has played a special and important role in
Banach Space Theory having its honour place in several textbooks related to
Banach Space Theory \cite{AK, DF, DJT, Pi, Ryan, Woy}; PDT has a strong
connection with the aforementioned inclusion results, as we explain below. In
fact, if $0<p<\infty,$ PDT states that for a given continuous linear operator
$u:X\rightarrow Y$ the following assertions are equivalent:

(i) There exists a $C>0$ so that
\[%
%TCIMACRO{\tsum \limits_{j=1}^{m}}%
%BeginExpansion
{\textstyle\sum\limits_{j=1}^{m}}
%EndExpansion
\left\Vert u(x_{j})\right\Vert ^{p}\leq C\sup_{\varphi\in B_{X^{\ast}}}%
%TCIMACRO{\tsum \limits_{j=1}^{m}}%
%BeginExpansion
{\textstyle\sum\limits_{j=1}^{m}}
%EndExpansion
\left\vert \varphi(x_{j})\right\vert ^{p}\text{ for every }m.
\]
(ii) There are a Borel probability measure $\mu$ on $B_{X^{\ast}}$ (with the
weak-star topology) and $C>0$ such that%
\begin{equation}
\left\Vert u(x)\right\Vert \leq C\left(  \int_{B_{X^{\ast}}}\left\vert
\varphi(x_{j})\right\vert ^{p}d\mu\right)  ^{\frac{1}{p}}. \label{p111}%
\end{equation}
Using the canonical inclusions between $L_{p}$ spaces we conclude that, if
$0<p\leq q<\infty,$ the inequality (\ref{p111}) implies that%
\[
\left\Vert u(x)\right\Vert \leq C\left(  \int_{B_{X^{\ast}}}\left\vert
\varphi(x_{j})\right\vert ^{q}d\mu\right)  ^{\frac{1}{q}}%
\]
and we obtain the implication (\ref{wqqq})$\Rightarrow$(\ref{wqqq2}) as a corollary.

Due to its strong importance in Banach Space Theory, PDT was re-discovered in
different contexts in the last decades (e.g. \cite{AMe, CP, Dimant, FaJo,
Geiss, SP, Anais, Muj}) and, since 2009, in \cite{BPR, BPRn, psjmaa} some
attempts were made in the direction of showing that one unique PDT can be
stated in such a general way that all the possible Pietsch Domination-type
theorems would be straightforward particular cases of this unified
Pietsch-Domination theorem.

Thus, the second contribution of this paper is to prove a ``full general
Pietsch Domination Theorem'' that, besides its own interest, we do believe
that will be useful to delimit the scope of Pietsch-type theorems. Some
connections with the recent promising notion of weighted summability
introduced in \cite{psmz} are traced.

\section{The Inclusion Principle}

In this section we deal with general values for $p_{j}$ and $q_{j}$ satisfying
(\ref{wer}). In order to be useful in different contexts, we state the result
in a very general form.

Let $X$, $Y,$ $Z$, $V$ and $W$ be (arbitrary) non-void sets. The set of all
mappings from $X$ to $Y$ will be represented by $Map(X,Y)$. Let $\mathcal{H}%
\subset$ $Map(X,Y)$ and
\begin{align*}
R\colon Z\times W  &  \longrightarrow\lbrack0,\infty),\text{ and }\\
S\colon\mathcal{H}\times Z\times V  &  \longrightarrow\lbrack0,\infty)
\end{align*}
be arbitrary mappings. If $1\leq p\leq q<\infty$, suppose that%
\[
\sup_{w\in W}\sum_{j=1}^{m}R\left(  z_{j},w\right)  ^{p}<\infty\text{ and
}\sup_{v\in V}\sum_{j=1}^{m}S(f,z_{j},v)^{q}<\infty
\]
for every positive integer $m$ and $z_{1},...,z_{m}\in Z$ (in most of the
applications $V$ and $W$ are compact spaces and $R$ and $S$ have some trace of
continuity to assure that both $\sup$ are finite). If $\alpha\in\mathbb{R}$,
we will say that $f\in\mathcal{H}$ is $RS$-abstract $((q,\alpha),p)$-summing
(notation $f\in RS_{((q,\alpha),p)}$) if there is a constant $C>0$ so that%
\begin{equation}
\left(  \sup_{v\in V}\sum_{j=1}^{m}S(f,z_{j},v)^{q}\right)  ^{\frac{1}{\alpha
}}\leq C\sup_{w\in W}\sum_{j=1}^{m}R\left(  z_{j},w\right)  ^{p},
\end{equation}
for all $z_{1},\ldots,z_{m}\in Z$ and $m$.

\begin{theorem}
[Inclusion Principle]\label{yio}If $p_{j}$ and $q_{j}$ satisfy (\ref{wer}),
then%
\[
RS_{\left(  \left(  q_{1},1\right)  ,p_{1}\right)  }\subset RS_{\left(
\left(  q_{2},\alpha\right)  ,p_{2}\right)  }\text{ }%
\]
for%
\[
\alpha=\frac{q_{2}p_{1}}{q_{1}p_{2}}\text{ if }p_{1}<p_{2}.
\]

\end{theorem}

\begin{proof}
Let $f\in RS_{((q_{1},1),p_{1})}$. There is a $C>0$ such that%
\begin{equation}
\sup_{v\in V}\sum_{j=1}^{m}S(f,z_{j},v)^{q_{1}}\leq C\sup_{w\in W}\sum
_{j=1}^{m}R\left(  z_{j},w\right)  ^{p_{1}},\label{yr}%
\end{equation}
for all $z_{1},\ldots,z_{m}\in Z$ and $m\in\mathbb{N}$. If each $\eta
_{1},...,\eta_{m}$ is a positive integer, by considering each $z_{j}$ repeated
$\eta_{j}$ times in (\ref{yr}) one can easily note that
\begin{equation}
\sup_{v\in V}\sum_{j=1}^{m}\eta_{j}S(f,z_{j},v)^{q_{1}}\leq C\sup_{w\in W}%
\sum_{j=1}^{m}\eta_{j}R\left(  z_{j},w\right)  ^{p_{1}},\label{yr2}%
\end{equation}
for all $z_{1},\ldots,z_{m}\in Z$ and $m\in\mathbb{N}$. Now, using a clever
argument credited to Mendel and Schechtman (used recently, in different
contexts, in \cite{FaJo, psmz, psjmaa}) we can conclude that (\ref{yr2}) holds
for arbitrary positive real numbers $\eta_{j}.$ The idea is to pass from
integers to rationals by \textquotedblleft cleaning\textquotedblright%
\ denominators and from rationals to real numbers using density.

Since $p_{1}<p_{2}$ we have $q_{1}<q_{2}.$ Define $p,q$ as
\[
\frac{1}{p}=\frac{1}{p_{1}}-\frac{1}{p_{2}}\text{ \ \ and \ \ }\frac{1}%
{q}=\frac{1}{q_{1}}-\frac{1}{q_{2}}.
\]
So we have $1\leq q\leq p<\infty;$ next, let $m\in%
%TCIMACRO{\U{2115} }%
%BeginExpansion
\mathbb{N}
%EndExpansion
$ and $z_{1},z_{2},...,z_{m}\in Z$ be fixed. For each $j=1,...,m$, consider
the map%
\begin{align*}
\lambda_{j}  &  :V\rightarrow\lbrack0,\infty)\\
\lambda_{j}(v)  &  :=S(f,z_{j},v)^{\frac{q_{2}}{q}}.
\end{align*}
Thus,
\begin{align*}
\lambda_{j}(v)^{q_{1}}S(f,z_{j},v)^{q_{1}}  &  =S(f,z_{j},v)^{\frac{q_{1}%
q_{2}}{q}}S(f,z_{j},v)^{q_{1}}\\
&  =S(f,z_{j},v)^{q_{2}}.
\end{align*}
Recalling that (\ref{yr2}) is valid for arbitrary positive real numbers
$\eta_{j},$ we get, for $\eta_{j}=\lambda_{j}(v)^{q_{1}}$,
\begin{align*}
\sum_{j=1}^{m}S(f,z_{j},v)^{q_{2}}  &  =\sum_{j=1}^{m}\lambda_{j}(v)^{q_{1}%
}S(f,z_{j},v)^{q_{1}}\\
&  \leq C\sup_{w\in W}\sum_{j=1}^{m}\lambda_{j}(v)^{q_{1}}R\left(
z_{j},w\right)  ^{p_{1}}%
\end{align*}
for every $v\in V$. Also, since $p,p_{2}>p_{1}$ and $\frac{1}{(p/p_{1})}%
+\frac{1}{(p_{2}/p_{1})}=1,$ invoking H\"{o}lder's Inequality we obtain
\begin{align*}
\sum_{j=1}^{m}S(f,z_{j},v)^{q_{2}}  &  \leq C\sup_{w\in W}\sum_{j=1}%
^{m}\lambda_{j}(v)^{q_{1}}R\left(  z_{j},w\right)  ^{p_{1}}\\
&  \leq C\sup_{w\in W}\left[  \left(  \sum_{j=1}^{m}\lambda_{j}(v)^{\frac
{q_{1}p}{p_{1}}}\right)  ^{\frac{p_{1}}{p}}\left(  \sum_{j=1}^{m}R\left(
z_{j},w\right)  ^{p_{2}}\right)  ^{\frac{p_{1}}{p_{2}}}\right] \\
&  =C\left(  \sum_{j=1}^{m}\lambda_{j}(v)^{\frac{q_{1}p}{p_{1}}}\right)
^{\frac{p_{1}}{p}}\sup_{w\in W}\left(  \sum_{j=1}^{m}R\left(  z_{j},w\right)
^{p_{2}}\right)  ^{\frac{p_{1}}{p_{2}}}%
\end{align*}
for every $v\in V$. Since $\frac{q_{1}p}{p_{1}}\geq p\geq q$ we have
$\left\Vert .\right\Vert _{\ell_{\frac{q_{1}p}{p_{1}}}}\leq\left\Vert
.\right\Vert _{\ell_{q}}$ and then%
\begin{align*}
\sum_{j=1}^{m}S(f,z_{j},v)^{q_{2}}  &  \leq C\left(  \sum_{j=1}^{m}\lambda
_{j}(v)^{q}\right)  ^{\frac{q_{1}}{q}}\sup_{w\in W}\left(  \sum_{j=1}%
^{m}R\left(  z_{j},w\right)  ^{p_{2}}\right)  ^{\frac{p_{1}}{p_{2}}}\\
&  =C\left(  \sum_{j=1}^{m}S(f,z_{j},v)^{q_{2}}\right)  ^{\frac{q_{1}}{q}}%
\sup_{w\in W}\left(  \sum_{j=1}^{m}R\left(  z_{j},w\right)  ^{p_{2}}\right)
^{\frac{p_{1}}{p_{2}}}%
\end{align*}
for every $v\in V.$ We thus have%
\[
\left(  \sum_{j=1}^{m}S(f,z_{j},v)^{q_{2}}\right)  ^{1-\frac{q_{1}}{q}}\leq
C\sup_{w\in W}\left(  \sum_{j=1}^{m}R\left(  z_{j},w\right)  ^{p_{2}}\right)
^{\frac{p_{1}}{p_{2}}}%
\]
for every $v\in V,$ and we can finally conclude that
\[
\left(  \sup_{v\in V}\sum_{j=1}^{m}S(f,z_{j},v)^{q_{2}}\right)  ^{\frac
{q_{1}p_{2}}{q_{2}p_{1}}}\leq C^{\frac{p_{2}}{p_{1}}}\sup_{w\in W}\sum
_{j=1}^{m}R\left(  z_{j},w\right)  ^{p_{2}}.
\]

\end{proof}

\begin{remark}
\label{remalpha} It is interesting to mention that as $q_{j}$ becomes closer
to $p_{j}$ for $j=1$ and $2,$ the value $\frac{q_{1}p_{2}}{q_{2}p_{1}}$
becomes closer to $1$(which occurs in the linear setting when $p_{j}=q_{j}$
for $j=1$ and $2$). In other words, the effect of the lack of linearity in our
estimates is weaker when $p_{j}$ and $q_{j}$ are closer and, in the extreme
case where $p_{1}=q_{1}$ and $p_{2}=q_{2},$ then $\alpha=1$ and we have a
\textquotedblleft perfect generalization\textquotedblright\ of the linear result.
\end{remark}

\section{Applications on the nonlinear absolutely summing operators}

\subsection{Absolutely summing operators: a brief summary}

In the real line it is well-known that a series is absolutely convergent
precisely when it is unconditionally convergent. For infinite-dimensional
Banach spaces it is easy to verify that the situation is different; for
example, for $\ell_{p}$ spaces with $1<p<\infty$, it is easy to construct an
unconditionally convergent series which fails to be absolutely convergent.
However the behavior for arbitrary Banach spaces was not known before 1950.
For $\ell_{1},$ for example, the construction is much more complicated (see M.
S. McPhail's work from 1947, \cite{Mac}).

This perspective leads to the feeling that this property (having an
unconditionally summable series which is not absolutely summable) could be
shared by all infinite-dimensional Banach-spaces. This question was raised by
S. Banach in his monograph \cite[page 40]{Banach32} and appears as Problem 122
in the Scottish Book (see \cite{Mau}).

In 1950, A. Dvoretzky and C. A. Rogers \cite{DR} solved this question by
showing that in every infinite-dimensional Banach space there is an
unconditionally convergent series which fails to be absolutely convergent.
This new panorama of the subject called the attention of A. Grothendieck who
provided, in his thesis \cite{Gro1955}, a different approach to the
Dvoretzky-Rogers result. His thesis, together with his R\'{e}sum\'{e}
\cite{Gro1953}, can be regarded as the beginning of the theory of absolutely
$(q,p)$-summing operators.

The notion of absolutely $(q,p)$-summing operator, as we know nowadays, is due
to B. Mitiagin and A. Pe\l czy\'{n}ski \cite{MPel} and A. Pietsch \cite{stu}.
Pietsch's paper is a classical and particular role is played by the Domination
Theorem, which presents an unexpected measure-theoretical characterization of
$p$-summing operators. The same task was brilliantly done, one year later, by
J. Lindenstrauss and A. Pe\l czy\'{n}ski's paper \cite{LP} which reformulated
Grothendieck's tensorial arguments giving birth to a comprehensible theory
with broad applications in Banach Space Theory.

From now on the space of all continuous linear operators from a Banach space
$X$ to a Banach space $Y$ will be denoted by $\mathcal{L}(X,Y).$ If $1\leq
p\leq q<\infty,$ we say that the Banach space operator $u:X\rightarrow Y$ is
$(q,p)$-summing if there is an induced operator
\[%
\begin{array}
[c]{ccccc}%
\hat{u} & : & \ell_{p}^{\text{weak}}(X) & \longrightarrow & \ell
_{q}^{\text{strong}}(Y)\\
&  & (x_{n})_{n=1}^{\infty} & \mapsto & (ux_{n})_{n=1}^{\infty}.\\
&  &  &  &
\end{array}
\]
Above $\ell_{p}^{\text{weak}}(X):=\{(x_{j})_{j=1}^{\infty}\subset
X:\sup_{\varphi\in B_{X^{\ast}}}(%
%TCIMACRO{\tsum \nolimits_{j}}%
%BeginExpansion
{\textstyle\sum\nolimits_{j}}
%EndExpansion
\left\vert \varphi(x_{j})\right\vert ^{p})^{1/p}<\infty\}.$ The class of
absolutely $(q,p)$-summing linear operators from $X$ to $Y$ will be
represented by $\Pi_{q,p}\left(  X,Y\right)  .$ For details on the linear
theory of absolutely summing operators we refer to the classical book
\cite{DJT}. The linear theory of absolutely summing operators was intensively
investigated in the 70's and several classical papers can tell the story (we
mention \cite{ben1,ben2,Carl,Dies,DPR,pisier} and the monograph \cite{DJT} for
a complete panorama).

Special role is played by Grothendieck's Theorem and Pietsch-Domination Theorem:

\begin{theorem}
[Grothendieck]Every continuous linear operator from $\ell_{1}$ to $\ell_{2}$
is absolutely $(1,1)$-summing.
\end{theorem}

\begin{theorem}
[Lindenstrauss and Pe\l czy\'{n}ski]\label{uyy}If $X$ and $Y$ are
infinite-dimensional Banach spaces, $X$ has an unconditional Schauder basis
and $\Pi_{1,1}(X,Y)=\mathcal{L}(X,Y)$ then $X=\ell_{1}$ and $Y$ is a Hilbert space.
\end{theorem}

\begin{theorem}
[Pietsch-Domination Theorem]\label{ppk}If $X$ and $Y$ are Banach spaces, a
continuous linear operator $T:X\rightarrow Y$ is absolutely $(p,p)$-summing if
and only if there is a constant $C>0$ and a Borel probability measure $\mu$ on
the closed unit ball of the dual of $X,$ $\left(  B_{X^{\ast}},\sigma(X^{\ast
},X)\right)  ,$ such that%
\[
\left\Vert T(x)\right\Vert \leq C\left(  \int_{B_{X^{\ast}}}\left\vert
\varphi(x)\right\vert ^{p}d\mu\right)  ^{\frac{1}{p}}.
\]
\bigskip
\end{theorem}

An immediate consequence of the Pietsch Domination Theorem is that, for $1\leq
r\leq s<\infty,$ every absolutely $(r,r)$-summing operator is absolutely
$(s,s)$-summing. However a more general result is valid. As mentioned in the
first section, this result is essentially due to Kwapie\'{n} (\cite{K22}):

\begin{theorem}
[Inclusion Theorem]\label{IT} If $X$ and $Y$ are Banach spaces and
\begin{equation}%
\begin{array}
[c]{c}%
p_{j}\leq q_{j}\text{ for }j=1,2,\\
1\leq p_{1}\leq p_{2}<\infty,\\
1\leq q_{1}\leq q_{2}<\infty,\\
\frac{1}{p_{1}}-\frac{1}{q_{1}}\leq\frac{1}{p_{2}}-\frac{1}{q_{2}},
\end{array}
\label{abod}%
\end{equation}
then%
\begin{equation}
\Pi_{q_{1},p_{1}}\left(  X,Y\right)  \subset\Pi_{q_{2},p_{2}}\left(
X,Y\right)  . \label{II}%
\end{equation}

\end{theorem}

The end of the 60's was also the time of the birth of the notion of
\emph{type} and \emph{cotype}. It probably began to be conceived in the
S\'{e}minaire Laurent Schwartz, and after important contributions by J.
Hoffmann-J\o rgensen \cite{HJ}, B. Maurey \cite{Ma2}, S. Kwapie\'{n}
\cite{K22}, and H. Rosenthal \cite{Ro}, the concept was formalized by B.
Maurey and G. Pisier \cite{pisier}.

Since B. Maurey and G. Pisier's seminal paper \cite{pisier}, the connection of
the notion of cotype and the concept of absolutely summing operators become
clear. In 1992, M. Talagrand \cite{T1} proved very deep results complementing
previous results of B. Maurey and G. Pisier showing that cotype $2$ spaces
have indeed a special behavior in the theory of absolutely summing operators:

\begin{theorem}
[Maurey-Pisier and Talagrand]If a Banach space $X$ has cotype $q$, then
$id_{X}$ is absolutely $(q,1)$-summing. The converse is true, except for $q=2$.
\end{theorem}

In the last two decades the interest of the theory was moved to the nonlinear
setting although there are still some challenging questions being investigated
in the linear setting (see \cite{unc, PellZ}). For example, recent results
from \cite{PellZ} complements the Lindenstrauss-Pe\l czy\'{n}ski Theorem
\ref{uyy} (below $\cot X$ denotes the infimum of the cotypes assumed by $X$):

\begin{theorem}
(\cite{PellZ}) Let $X$ and $Y$ be infinite-dimensional Banach spaces.

(i) If $\Pi_{1,1}(X,Y)=\mathcal{L}(X,Y)$ then $\cot X=\cot Y=2$.

(ii) If $2\leq r<\cot Y$ and $\Pi_{q,r}(X,Y)=\mathcal{L}(X,Y),$ then
$\mathcal{L}(\ell_{1},\ell_{\cot Y})=\Pi_{q,r}(\ell_{1},\ell_{\cot Y})$.
\end{theorem}

The extension of the classical linear theory of absolutely summing operators
to the multilinear setting is very far from being a mere exercise of
generalization with expected results obtained by induction. In fact, some
multilinear approaches are simple but there are several delicate questions
related to the multilinear extensions of absolutely summing operators. Some
illustrative examples and applications can be seen in \cite{Acosta, ag, Def2,
Perr, ppp}). For non-multilinear approaches we refer to \cite{BBJP, Junek,
Nach, MP}).

The advance of the nonlinear theory of absolutely summing operators leads to
the search for nonlinear versions of the Pietsch Domination-Factorization
Theorem (see, for example, \cite{AMe, JFA, BPR, FaJo, Geiss, SP}). Recently,
in \cite{BPRn} (see also an addendum in \cite{psjmaa} and \cite{psmz} for a
related result), an abstract unified approach to Pietsch-type results was
presented as an attempt to show that all the known Pietsch-type theorems were
particular cases of a unified general version. However, these approaches were
not complete, as we will show later.

\subsection{Applications to the theory of absolutely summing multilinear
operators}

The multilinear theory of absolutely summing mappings seems to have its
starting point in \cite{BH,LLL} but only in the 1980's it gained more
attention, motivated by A. Pietsch's work \cite{PPPP}; recently some nice
results and applications have appeared, mainly related to the notion of fully
or multiple summability (see \cite{Acosta, BBJP2, Na, Def, Def2} and
references therein). This section will actually show that for multilinear
mappings there exists an improved version of the Inclusion Principle (we just
need to explore the multi-linearity).

For technical reasons the present abstract setting is slightly different from
the one of the previous section. Let $X$, $Y,$ $V,$ $G,$ $W$ be (arbitrary)
non-void sets, $Z$ a vector space and $\mathcal{H}\subset Map(X,Y)$. Consider
the arbitrary mappings
\begin{align*}
R\colon Z\times G\times W  &  \longrightarrow\lbrack0,\infty)\\
S\colon\mathcal{H}\times Z\times G\times V  &  \longrightarrow\lbrack
0,\infty).
\end{align*}
Let $1\leq p\leq q<\infty$ and $\alpha\in\mathbb{R}$. Suppose that%
\[
\sup_{w\in W}\sum_{j=1}^{m}R\left(  z_{j},g_{j},w\right)  ^{p}<\infty\text{
and }\sup_{v\in V}\sum_{j=1}^{m}S(f,z_{j},g_{j},v)^{q}<\infty
\]
for every positive integer $m$ and $z_{1},...,z_{m}\in Z$ and $g_{1}%
,...,g_{m}\in G.$ We will say that $f\in\mathcal{H}$ is $(q,p)$-abstract
$(R,S)$-summing (notation $f\in RS_{(q,p)}$) if there is a constant $C>0$ so
that%
\begin{equation}
\left(  \sup_{v\in V}\sum_{j=1}^{m}S(f,z_{j},g_{j},v)^{q}\right)  ^{\frac
{1}{q}}\leq C\left(  \sup_{w\in W}\sum_{j=1}^{m}R\left(  z_{j},g_{j},w\right)
^{p}\right)  ^{1/p},
\end{equation}
for all $z_{1},\ldots,z_{m}\in Z,$ $g_{1},...,g_{m}\in G$ and $m\in\mathbb{N}%
$. We will say that $S$ and $R$ are multiplicative in the variable $Z$ if%
\begin{align*}
R\left(  \lambda z,g,w\right)   &  =\left\vert \lambda\right\vert R\left(
z,g,w\right)  ,\\
S\left(  f,\lambda z,g,v\right)   &  =\left\vert \lambda\right\vert S\left(
f,z,g,v\right)  .
\end{align*}

\begin{theorem}
\label{red}Let $p_{j}$ and $q_{j}$ be as in (\ref{abod}) and suppose that $S$
and $R$ are multiplicative in the variable $Z.$ Then%
\[
RS_{(q_{1},p_{1})}\subset RS_{\left(  q_{2},p_{2}\right)  }.\text{ }%
\]

\end{theorem}

\begin{proof}
If $p_{1}=p_{2}=p$ the result is clear. So, let us consider $p_{1}<p_{2}$ (and
hence $q_{1}<q_{2}$)$.$ If $f\in RS_{(q_{1},p_{1})},$ \ there is a $C>0$ such
that%
\begin{equation}
\left(  \sup_{v\in V}\sum_{j=1}^{m}S(f,z_{j},g_{j},v)^{q_{1}}\right)
^{\frac{1}{q_{1}}}\leq C\sup_{w\in W}\left(  \sum_{j=1}^{m}R\left(
z_{j},g_{j},w\right)  ^{p_{1}}\right)  ^{\frac{1}{p_{1}}}, \label{yrt}%
\end{equation}
for all $z_{1},\ldots,z_{m}\in Z,$ $g_{1},...,g_{m}\in G$ and $m\in\mathbb{N}
$. Then%
\begin{equation}
\left(  \sup_{v\in V}\sum_{j=1}^{m}S(f,\lambda_{j}z_{j},g_{j},v)^{q_{1}%
}\right)  ^{\frac{1}{q_{1}}}\leq C\sup_{w\in W}\left(  \sum_{j=1}^{m}R\left(
\lambda_{j}z_{j},g_{j},w\right)  ^{p_{1}}\right)  ^{\frac{1}{p_{1}}},
\end{equation}
for all $z_{1},\ldots,z_{m}\in Z,$ $\lambda_{1},...,\lambda_{m}\in\mathbb{K}$,
$g_{1},...,g_{m}\in G$ and $m\in\mathbb{N}$. Define $p,q$ by
\[
\frac{1}{p}=\frac{1}{p_{1}}-\frac{1}{p_{2}}\text{ \ \ and \ \ }\frac{1}%
{q}=\frac{1}{q_{1}}-\frac{1}{q_{2}}.
\]
So we have $1\leq q\leq p<\infty;$ let $m\in%
%TCIMACRO{\U{2115} }%
%BeginExpansion
\mathbb{N}
%EndExpansion
,$ $z_{1},z_{2},...,z_{m}\in Z$ and $g_{1},...,g_{m}\in G$ be fixed. For each
$j=1,...,m$, consider%
\begin{align*}
\lambda_{j}  &  :V\rightarrow\lbrack0,\infty)\\
\lambda_{j}(v)  &  :=S(f,z_{j},g_{j},v)^{\frac{q_{2}}{q}}.
\end{align*}
So, recalling that $S$ is multiplicative in $Z$, we have%
\begin{align*}
\left(  \sum_{j=1}^{m}S(f,z_{j},g_{j},v)^{q_{2}}\right)  ^{\frac{1}{q_{1}}}
&  =\left(  \sum_{j=1}^{m}S(f,\lambda_{j}(v)z_{j},g_{j},v)^{q_{1}}\right)
^{\frac{1}{q_{1}}}\\
&  \leq C\sup_{w\in W}\left(  \sum_{j=1}^{m}R\left(  \lambda_{j}(v)z_{j}%
,g_{j},w\right)  ^{p_{1}}\right)  ^{\frac{1}{p_{1}}}%
\end{align*}
for every $v\in V$. Since $R$ is multiplicative in $Z$ and, as we did before,
from H\"{o}lder's Inequality we obtain
\begin{align*}
\left(  \sum_{j=1}^{m}S(f,z_{j},g_{j},v)^{q_{2}}\right)  ^{\frac{1}{q_{1}}}
&  \leq C\sup_{w\in W}\left(  \sum_{j=1}^{m}\lambda_{j}(v)^{p_{1}}R\left(
z_{j},g_{j},w\right)  ^{p_{1}}\right)  ^{\frac{1}{p_{1}}}\\
&  \leq C\sup_{w\in W}\left[  \left(  \sum_{j=1}^{m}\lambda_{j}(v)^{p}\right)
^{\frac{p_{1}}{p}}\left(  \sum_{j=1}^{m}R\left(  z_{j},g_{j},w\right)
^{p_{2}}\right)  ^{\frac{p_{1}}{p_{2}}}\right]  ^{\frac{1}{p_{1}}}\\
&  =C\left(  \sum_{j=1}^{m}\lambda_{j}(v)^{p}\right)  ^{\frac{1}{p}}\sup_{w\in
W}\left(  \sum_{j=1}^{m}R\left(  z_{j},g_{j},w\right)  ^{p_{2}}\right)
^{\frac{1}{p_{2}}}%
\end{align*}
for every $v\in V$. Since $p\geq q$ we have $\left\Vert .\right\Vert
_{\ell_{p}}\leq\left\Vert .\right\Vert _{\ell_{q}}$ and then%
\begin{align*}
\left(  \sum_{j=1}^{m}S(f,z_{j},g_{j},v)^{q_{2}}\right)  ^{\frac{1}{q_{1}}}
&  \leq C\left(  \sum_{j=1}^{m}\lambda_{j}(v)^{q}\right)  ^{\frac{1}{q}}%
\sup_{w\in W}\left(  \sum_{j=1}^{m}R\left(  z_{j},g_{j},w\right)  ^{p_{2}%
}\right)  ^{\frac{1}{p_{2}}}\\
&  =C\left(  \sum_{j=1}^{m}S(f,z_{j},g_{j},v)^{q_{2}}\right)  ^{\frac{1}{q}%
}\sup_{w\in W}\left(  \sum_{j=1}^{m}R\left(  z_{j},g_{j},w\right)  ^{p_{2}%
}\right)  ^{\frac{1}{p_{2}}}%
\end{align*}
for every $v\in V$ and we easily conclude the proof.
\end{proof}

Let us show how the above result applies to the multilinear theory of
absolutely summing mappings. Our intention is illustrative rather than
exhaustive. From now on we will use the notation $\mathcal{L}(X_{1}%
,...,X_{n};Y)$ to represent the spaces of continuous $n$-linear mappings from
$X_{1}\times\cdots\times X_{n}$ to $Y$. For the theory of multilinear mappings
between Banach spaces we refer to \cite{Din, Mujica}. Consider the following
concepts of multilinear summability for $1\leq p\leq q<\infty$ (inspired in
\cite{CP, Dimant}):

\begin{itemize}
\item[\textbf{1.-)}] A mapping $T\in\mathcal{L}(X_{1},...,X_{n};Y)$ is
$(q,p)$-semi integral if there exists $C\geq0$ such that%
\begin{equation}
\left(  \sum\limits_{j=1}^{m}\parallel T(x_{j}^{1},...,x_{j}^{n})\parallel
^{q}\right)  ^{1/q}\leq C\left(  \underset{\varphi_{l}\in B_{X_{l}^{\ast}%
},l=1,...,n}{\sup}\sum\limits_{j=1}^{m}\mid\varphi_{1}(x_{j}^{1}%
)...\varphi_{n}(x_{j}^{n})\mid^{p}\right)  ^{1/p} \label{day}%
\end{equation}
for every $m\in\mathbb{N}$, $x_{j}^{l}\in X_{l}$ with $l=1,...,n$ and
$j=1,...,m.$ In the above situation we write $T\in\mathcal{L}_{si(q,p)}%
(X_{1},...,X_{n};Y)$).

\item[\textbf{2.-)}] A mapping $T\in\mathcal{L}(X_{1},...,X_{n};Y)$ is
strongly $(q,p)$-summing if there exists $C\geq0$ such that%
\[
\left(  \sum\limits_{j=1}^{m}\parallel T(x_{j}^{1},...,x_{j}^{n})\parallel
^{q}\right)  ^{1/q} \leq C\left(  \underset{\varphi\in B_{\mathcal{L}%
(X_{1},...,X_{n};\mathbb{K})}}{\sup}\sum\limits_{j=1}^{m}\mid\varphi(x_{j}%
^{1},...,x_{j}^{n})\mid^{p}\right)  ^{1/p}%
\]
for every $m\in\mathbb{N}$, $x_{j}^{l}\in X_{l}$ with $l=1,...,n$ and
$j=1,...,m.$ In the above situation we write $T\in\mathcal{L}_{ss(q,p)}%
(X_{1},...,X_{n};Y)$.
\end{itemize}

For both concepts there is a natural Pietsch-Domination-type theorem (see
\cite{CP,Dimant}) and as a corollary the following inclusion results hold:

\begin{proposition}
If $1\leq p\leq q<\infty$, then, for any Banach spaces $X_{1},...,X_{n},Y$,
the following inclusions hold:
\begin{align*}
\mathcal{L}_{si(p,p)}(X_{1},...,X_{n};Y)  &  \subset\mathcal{L}_{si(q,q)}%
(X_{1},...,X_{n};Y)\text{ and}\\
\mathcal{L}_{ss(p,p)}(X_{1},...,X_{n};Y)  &  \subset\mathcal{L}_{ss(q,q)}%
(X_{1},...,X_{n};Y).
\end{align*}

\end{proposition}

However, the Pietsch Domination Theorem is useless for the other choices of
$p_{j},q_{j}.$ But, as it will be shown, in this case the multilinearity
allows us to obtain better results than those from Theorem \ref{yio}.

For the class of semi-integral mappings we may choose $Z=X_{1}$,
$G=X_{2}\times\cdots\times X_{n}$, $W=B_{X_{1}^{\ast}}\times\cdots\times
B_{X_{n}^{\ast}},$ $V=\{0\},$ $\mathcal{H}=\mathcal{L}(X_{1},...,X_{n};Y)$ and
consider the mappings
\begin{align*}
R\colon Z\times G\times W  &  \longrightarrow\lbrack0,\infty)\\
R\left(  x_{1},(x_{2}...,x_{n}),(\varphi_{1},...,\varphi_{n})\right)   &
=\mid\varphi_{1}(x_{1})...\varphi_{n}(x_{n})\mid
\end{align*}
and%
\begin{align*}
S\colon\mathcal{H}\times Z\times G\times V  &  \longrightarrow\lbrack
0,\infty)\\
S\left(  T,x_{1},(x_{2}...,x_{n}),0\right)   &  =\parallel T(x_{1}%
,...,x_{n})\parallel.
\end{align*}

The case of the class of strongly summing multilinear mappings is analogous.
So, as a consequence of Theorem \ref{red}, we have:

\begin{proposition}
If $p_{j}$ and $q_{j}$ are as in (\ref{abod}) then, for any Banach spaces
$X_{1},...,X_{n},Y$, the following inclusions hold:
\begin{align*}
\mathcal{L}_{si(q_{1},p_{1})}(X_{1},...,X_{n};Y)  &  \subset\mathcal{L}%
_{si(q_{2},p_{2})}(X_{1},...,X_{n};Y)\text{ and}\\
\mathcal{L}_{ss(q_{1},p_{1})}(X_{1},...,X_{n};Y)  &  \subset\mathcal{L}%
_{ss(q_{2},p_{2})}(X_{1},...,X_{n};Y).
\end{align*}

\end{proposition}

\subsection{Applications to non-multilinear absolutely summing operators}

As in the previous section, we intend to illustrate how the Inclusion
Principle can be invoked in other situations; we have no exhaustive purpose.

Let us consider the following definitions extending the notion of
semi-integral and strongly multilinear mappings to the non-multilinear
context, even with spaces having a less rich structure than a Banach space:

\begin{definition}
Let $X_{1},...,X_{n}$ be normed spaces and $Y=(Y,d)$ be a metric space. An
arbitrary map $f:X_{1}\times\cdots\times X_{n}\rightarrow Y$ is $((q,\alpha
),p)$-semi integral at $(a_{1},...,a_{n})\in X_{1}\times\cdots\times X_{n}$
(notation $f\in Map_{si((q,\alpha),p))}(X_{1},...,X_{n};Y)$) if there exists
$C\geq0$ such that%
\begin{align*}
&  \left(  \sum\limits_{j=1}^{m}\left(  d\left(  f(a_{1}+x_{j}^{1}%
,...,a_{n}+x_{j}^{n}),f(a_{1},...,a_{n})\right)  \right)  ^{q}\right)
^{1/\alpha}\\
&  \leq C\underset{\varphi_{l}\in B_{X_{l}^{\ast}},l=1,...,n}{\sup}%
\sum\limits_{j=1}^{m}\mid\varphi_{1}(x_{j}^{1})...\varphi_{n}(x_{j}^{n}%
)\mid^{p}%
\end{align*}
for every $m\in\mathbb{N}$, $x_{j}^{l}\in X_{l}$ with $l=1,...,n$ and
$j=1,...,m.$
\end{definition}

\begin{definition}
Let $X_{1},...,X_{n}$ be normed spaces and $Y=(Y,d)$ be a metric space. An
arbitrary map $f:X_{1}\times\cdots\times X_{n}\rightarrow Y$ is strongly
$((q,\alpha),p)$-summing at $(a_{1},...,a_{n})\in X_{1}\times\cdots\times
X_{n}$ (notation $f\in Map_{ss((q,\alpha),p))}(X_{1},...,X_{n};Y)$) if there
exists $C\geq0$ such that%
\begin{align*}
&  \left(  \sum\limits_{j=1}^{m}\left(  d\left(  f(a_{1}+x_{j}^{1}%
,...,a_{n}+x_{j}^{n}),f(a_{1},...,a_{n})\right)  \right)  ^{q}\right)
^{1/\alpha}\\
&  \leq C\underset{\varphi\in\mathcal{L}(X_{1},...,X_{n};\mathbb{K})}{\sup
}\sum\limits_{j=1}^{m}\mid\varphi(x_{j}^{1},...,x_{j}^{n})\mid^{p}%
\end{align*}
for every $m\in\mathbb{N}$, $x_{j}^{l}\in X_{l}$ with $l=1,...,n$ and
$j=1,...,m.$
\end{definition}

By choosing adequate parameters in Theorem \ref{yio} we obtain:

\begin{theorem}
If $p_{j}$ and $q_{j}$ satisfy (\ref{abod}), then%
\begin{align*}
Map_{si((q_{1},1),p_{1})}(X_{1},...,X_{n};Y) &  \subset Map_{si((q_{2}%
,\alpha),p_{2}))}(X_{1},...,X_{n};Y)\text{ and}\\
Map_{ss((q_{1},1),p_{1})}(X_{1},...,X_{n};Y) &  \subset Map_{ss((q_{2}%
,\alpha),p_{2}))}(X_{1},...,X_{n};Y)\text{ }%
\end{align*}
for%
\[
\alpha=\frac{q_{2}p_{1}}{q_{1}p_{2}}\text{ if }p_{1}<p_{2}.\text{\ }%
\]

\end{theorem}

\subsection{Applications to non-multilinear absolutely summing operators in
the sense of Matos}

In \cite{Nach} M. Matos considered a concept of summability which can be
characterized by means of an inequality as follows:

If $X$ and $Y$ are Banach spaces, a map $f:X\rightarrow Y$ is \emph{absolutely
$(q,p)$-summing at $a$} if there are constants $C>0$ and $\delta>0 $ such that%
\[
\sum_{j=1}^{\infty}\left\Vert f(a+z_{j})-f(a)\right\Vert ^{q}\leq
C\sup_{\varphi\in B_{X^{\ast}}}\sum_{j=1}^{\infty}\left\vert \varphi
(z_{j})\right\vert ^{p},
\]
for all $(z_{j})_{j=1}^{\infty}\in\ell_{p}^{u}(X)$ and
\[
\left\Vert (z_{j})_{j=1}^{\infty}\right\Vert _{w,p}:=\sup_{\varphi\in
B_{X^{\ast}}}\left(  {\sum\limits_{j=1}^{\infty}}\left\vert \varphi
(z_{j})\right\vert ^{p}\right)  ^{1/p}<\delta.
\]

Above,
\[
\ell_{p}^{u}(X):=\left\{  (z_{j})_{j=1}^{\infty}\in\ell_{p}^{\text{weak}%
}(X);\lim_{n\rightarrow\infty}\left\Vert (z_{j})_{j=n}^{\infty}\right\Vert
_{w,p}=0\right\}  .
\]

It is worth mentioning that there exists a version of our inclusion principle
in this context. If $\alpha\in\mathbb{R}$, we will say that $f:X\rightarrow Y$
is \emph{Matos absolutely $((q,\alpha),p)$-summing at $a$} (denoted by $f\in
M_{((q,\alpha),p)}$) if there are constants $C>0$ and $\delta>0$ such that
\begin{equation}
\left(  \sum_{j=1}^{\infty}\left\Vert f(a+z_{j})-f(a)\right\Vert ^{q}\right)
^{\frac{1}{\alpha}}\leq C\sup_{\varphi\in B_{X^{\ast}}}\sum_{j=1}^{\infty
}\left\vert \varphi(z_{j})\right\vert ^{p}, \label{33M2}%
\end{equation}
for all $(z_{j})_{j=1}^{\infty}\in\ell_{p}^{u}(X)$ and $\left\Vert
(z_{j})_{j=1}^{\infty}\right\Vert _{w,p}<\delta$. If $\alpha=1$ we recover
Matos' original concept and simply write $(q,p)$ instead of $((q,1),p)$.

With this at hand, we can now state the following result:

\begin{theorem}
\label{yio2}If $p_{j}$ and $q_{j}$ are as in \eqref{abod}, then
\[
M_{(q_{1},p_{1})}\subset M_{\left(  \left(  q_{2},\alpha\right)
,p_{2}\right)  }\text{ }%
\]
for%
\[
\alpha=\frac{q_{2}p_{1}}{q_{1}p_{2}}%
\]
whenever $p_{1}<p_{2}$.
\end{theorem}

\section{A full general version of the Pietsch Domination Theorem\label{fgg}}

If $X_{1},...,X_{n},Y$ are Banach spaces, the set of all continuous $n$-linear
mappings $T:X_{1}\times\cdots\times X_{n}\rightarrow Y$ is represented by
$\mathcal{L}(X_{1},...,X_{n};Y)$.\ All measures considered in this paper will
be probability measures defined in the Borel sigma-algebras of compact
topological spaces.

In this section, and for the sake of completeness, we will recall the more
general version that we know, until now, for the Pietsch Domination Theorem.
This approach is a combination of \cite{BPRn} and a recent improvement from
\cite{psjmaa} and will be generalized in the subsequent section.

Let $X$, $Y$ and $E$ be (arbitrary) non-void sets, $\mathcal{H}$ be a family
of mappings from $X$ to $Y$, $G$ be a Banach space and $K$ be a compact
Hausdorff topological space. Let
\[
R\colon K\times E\times G\longrightarrow\lbrack0,\infty)~\text{and}%
\mathrm{~}S\colon{\mathcal{H}}\times E\times G\longrightarrow\lbrack0,\infty)
\]
be mappings so that the following property hold:

\begin{quote}
``The mapping
\[
R_{x,b}\colon K\longrightarrow\lbrack0,\infty)~\text{defined by}%
~R_{x,b}(\varphi)=R(\varphi,x,b)
\]
is continuous for every $x\in E$ and $b\in G$.''
\end{quote}

Let $R$ and $S$ be as above and $0<p<\infty$. A mapping $f\in\mathcal{H}$ is
said to be $R$-$S$-abstract $p$-summing if there is a constant $C>0$ so that%
\begin{equation}
\left(  \sum_{j=1}^{m}S(f,x_{j},b_{j})^{p}\right)  ^{\frac{1}{p}}\leq
C\sup_{\varphi\in K}\left(  \sum_{j=1}^{m}R\left(  \varphi,x_{j},b_{j}\right)
^{p}\right)  ^{\frac{1}{p}}, \label{33M}%
\end{equation}
for all $x_{1},\ldots,x_{m}\in E,$ $b_{1},\ldots,b_{m}\in G$ and
$m\in\mathbb{N}$.

The general unified PDT reads as follows:

\begin{theorem}
(\cite{BPRn, psjmaa}) Let $R$ and $S$ be as above, $0<p<\infty$ and
$f\in{\mathcal{H}}$. Then $f$ is $R$-$S$-abstract $p$-summing if and only if
there is a constant $C>0$ and a Borel probability measure $\mu$ on $K$ such
that%
\begin{equation}
S(f,x,b)\leq C\left(  \int_{K}R\left(  \varphi,x,b\right)  ^{p}d\mu\right)
^{\frac{1}{p}} \label{olk}%
\end{equation}
for all $x\in E$ and $b\in G.$
\end{theorem}

From now on, if $X_{1},...,X_{n},Y$ are arbitrary sets, $Map(X_{1}%
,...,X_{n};Y)$ will denote the set of all arbitrary mappings from $X_{1}%
\times\cdots\times X_{n}$ to $Y$ (no assumption is necessary).

Let $0<q_{1},...,q_{n}<\infty$ and $1/q=%
%TCIMACRO{\tsum \limits_{j=1}^{n}}%
%BeginExpansion
{\textstyle\sum\limits_{j=1}^{n}}
%EndExpansion
1/q_{j}.$ A map $f\in Map(X_{1},...,X_{n};Y)$ is $(q_{1},...,q_{n})$-dominated
at $(a_{1},...,a_{n})\in X_{1}\times\cdots\times X_{n}$ if there is a $C>0$
and there are Borel probabilities $\mu_{k}$ on $B_{X_{k}^{\ast}},$
$k=1,...,n$, such that%
\begin{equation}
\left\Vert f(a_{1}+x^{(1)},...,a_{n}+x^{(n)})-f(a_{1},...,a_{n})\right\Vert
\leq C%
%TCIMACRO{\dprod \limits_{k=1}^{n}}%
%BeginExpansion
{\displaystyle\prod\limits_{k=1}^{n}}
%EndExpansion
\left(  \int_{B_{X_{k}^{\ast}}}\left\vert \varphi(x^{(k)})\right\vert ^{q_{k}%
}d\mu_{k}\right)  ^{\frac{1}{q_{k}}} \label{domGG}%
\end{equation}
for all $x^{(j)}\in X_{j}$, $j=1,...,n$.

In our recent note \cite{psmz} we observed that the general approach from
\cite{BPRn, psjmaa} was not able to characterize the mappings satisfying
(\ref{domGG}), and a new Pietsch-type theorem was proved:

\begin{theorem}
\label{ttta}(\cite{psmz})A map $f\in Map(X_{1},...,X_{n};Y)$ is $(q_{1}%
,...,q_{n})$-dominated at $(a_{1},...,a_{n})\in X_{1}\times\cdots\times X_{n}$
if and only if there is a $C>0$ such that
\begin{align}
&  \left(
%TCIMACRO{\dsum \limits_{j=1}^{m}}%
%BeginExpansion
{\displaystyle\sum\limits_{j=1}^{m}}
%EndExpansion
\left(  \left\vert b_{j}^{(1)}...b_{j}^{(n)}\right\vert \left\Vert
f(a_{1}+x_{j}^{(1)},...,a_{n}+x_{j}^{(n)})-f(a_{1},...,a_{n})\right\Vert
\right)  ^{q}\right)  ^{1/q}\label{qww}\\
&  \leq C%
%TCIMACRO{\dprod \limits_{k=1}^{n}}%
%BeginExpansion
{\displaystyle\prod\limits_{k=1}^{n}}
%EndExpansion
\sup_{\varphi\in B_{X_{k}^{\ast}}}\left(
%TCIMACRO{\dsum \limits_{j=1}^{m}}%
%BeginExpansion
{\displaystyle\sum\limits_{j=1}^{m}}
%EndExpansion
\left(  \left\vert b_{j}^{(k)}\right\vert \left\vert \varphi(x_{j}%
^{(k)})\right\vert \right)  ^{q_{k}}\right)  ^{1/q_{k}}\nonumber
\end{align}
for every positive integer $m$, $(x_{j}^{(k)},b_{j}^{(k)})\in X_{k}%
\times\mathbb{K}$, with $(j,k)\in\{1,...,m\}\times\{1,...,n\}.$
\end{theorem}

As pointed in \cite{psmz}, inequality (\ref{qww}) arises the curious idea of
weighted summability: each $x_{j}^{(k)}$ is interpreted as having a ``weight''
$b_{j}^{(k)}$ and in this context the respective sum
\[
\left\Vert f(a_{1}+x_{j}^{(1)},...,a_{n}+x_{j}^{(n)})-f(a_{1},...,a_{n}%
)\right\Vert
\]
inherits a weight $\left\vert b_{j}^{(1)} \cdot\dots\cdot b_{j}^{(n)}%
\right\vert $.

As it is shown in \cite{BPRn}, the unified PDT (UPDT) immediately recovers
several known Pietsch-type theorems. However, in at least one important
situation (the PDT for dominated multilinear mappings), the respective PDT is
not straightforwardly obtained from the UPDT from \cite{BPRn}. In fact, as
pointed in \cite{psmz}, the structural difference between (\ref{olk}) and
(\ref{domGG}) is an obstacle to recover some domination theorems as Theorem
\ref{ttta}. The same deficiency of the (general) UPDT\ will be clear in
Section \ref{ko}.

In the next section the approach of \cite{psmz} is translated to a more
abstract setting and the final result shows that Theorem \ref{ttta} holds in a
very general context. Some applications are given in order to show the reach
of this generalization.

\subsection{The full general Pietsch Domination Theorem}

In this section we prove a quite general PDT which seems to delimit the
possibilities of such kind of result. The procedure is an abstraction of the
main result of \cite{psmz}. It is curious the fact that the Unified Pietsch
Domination Theorem from \cite{BPRn} does not use Pietsch's original argument,
but this more general version, as in \cite{psmz}, uses precisely Pietsch's
original approach in an abstract disguise.

The main tool of our argument, as in Pietsch's original proof of the linear
case, is a Lemma by Ky Fan.

\begin{lemma}
[Ky Fan]Let $K$ be a compact Hausdorff topological space and $\mathcal{F}$ be
a concave family of functions $f:K\rightarrow\mathbb{R}$ which are convex and
lower semicontinuous. If for each $f\in\mathcal{F}$ there is a $x_{f}\in K$ so
that $f(x_{f})\leq0,$ then there is a $x_{0}\in K$ such that $f(x_{0})\leq0$
for every $f\in\mathcal{F}$ .
\end{lemma}

Let $X_{1},...,X_{n}$, $Y$ and $E_{1},...,E_{r}$ be (arbitrary) non-void sets,
$\mathcal{H}$ be a family of mappings from $X_{1}\times\cdots\times X_{n}$ to
$Y$ . Let also $K_{1},..,K_{t}$ be compact Hausdorff topological spaces,
$G_{1},...,G_{t}$ be Banach spaces and suppose that the maps%
\[
\left\{
\begin{array}
[c]{l}%
R_{j}\colon K_{j}\times E_{1}\times\cdots\times E_{r}\times G_{j}%
\longrightarrow\lbrack0,\infty)\text{, }j=1,...,t\\
S\colon{\mathcal{H}}\times E_{1}\times\cdots\times E_{r}\times G_{1}%
\times\cdots\times G_{t}\longrightarrow\lbrack0,\infty)
\end{array}
\right.
\]
satisfy:

\noindent\textbf{(1)} For each $x^{(l)}\in E_{l}$ and $b\in G_{j}$, with
$(j,l)\in\{1,...,t\}\times\{1,...,r\}$ the mapping%
\[
\left(  R_{j}\right)  _{x^{(1)},...,x^{(r)},b}\colon K_{j}\longrightarrow
\lbrack0,\infty)~\text{defined by }~\left(  R_{j}\right)  _{x^{(1)}%
,...,x^{(r)},b}(\varphi)=R_{j}(\varphi,x^{(1)},...,x^{(r)},b)
\]
is continuous.\newline\noindent\textbf{(2) }The following inequalities hold:%
\begin{equation}
\left\{
\begin{array}
[c]{l}%
R_{j}(\varphi,x^{(1)},...,x^{(r)},\eta_{j}b^{(j)})\leq\eta_{j}R_{j}\left(
\varphi,x^{(1)},...,x^{(r)},b^{(j)}\right)  \text{ }\\
S(f,x^{(1)},...,x^{(r)},\alpha_{1}b^{(1)},...,\alpha_{t}b^{(t)})\geq\alpha
_{1}...\alpha_{t}S(f,x^{(1)},...,x^{(r)},b^{(1)},...,b^{(t)})
\end{array}
\right.  \label{novaq}%
\end{equation}
for every $\varphi\in K_{j},x^{(l)}\in E_{l}$ (with $l=1,...,r$)$,0\leq
\eta_{j},\alpha_{j}\leq1,$ $b_{j}\in G_{j},$ with $j=1,...,t$ and
$f\in{\mathcal{H}}$.

\begin{definition}
\label{quatro}If $0<p_{1},...,p_{t},p<\infty,$ with\textrm{\textrm{\textrm{\ }%
}}$\frac{1}{p}=\frac{1}{p_{1}}+\cdots+\frac{1}{p_{t}}$, a mapping
$f:X_{1}\times\cdots\times X_{n}\rightarrow Y$ in $\mathcal{H}$ is said to be
$R_{1},...,R_{t}$-$S$-abstract $(p_{1},...,p_{t})$-summing if there is
a\textrm{\textrm{\textrm{\ }}}constant $C>0$ so that%
\begin{equation}
\left(  \sum_{j=1}^{m}S(f,x_{j}^{(1)},...,x_{j}^{(r)},b_{j}^{(1)}%
,...,b_{j}^{(t)})^{p}\right)  ^{\frac{1}{p}}\leq C{\displaystyle\prod
\limits_{k=1}^{t}}\sup_{\varphi\in K_{k}}\left(  \sum_{j=1}^{m}R_{k}\left(
\varphi,x_{j}^{(1)},...,x_{j}^{(r)},b_{j}^{(k)}\right)  ^{p_{k}}\right)
^{\frac{1}{p_{k}}} \label{cam-errado}%
\end{equation}
for all $x_{1}^{(s)},\ldots,x_{m}^{(s)}\in E_{s},$ $b_{1}^{(l)},\ldots
,b_{m}^{(l)}\in G_{l},$ $m\in\mathbb{N}$ and $(s,l)\in\{1,...,r\}\times
\{1,...,t\}$.
\end{definition}

The proof mimics the steps of the particular case proved in \cite{psmz}, and
hence we omit some details. Due the more abstract environment, the new proof
has extra technicalities but just in the final part of the proof a more
important care will be needed when dealing with the parameter $\beta.$

As in the proof of \cite{psmz}, we need the following lemma (see \cite[Page
17]{Hardy}):

\begin{lemma}
\label{yy}Let $0<p_{1},...,p_{n},p<\infty$ be so that $1/p=%
%TCIMACRO{\tsum \limits_{j=1}^{n}}%
%BeginExpansion
{\textstyle\sum\limits_{j=1}^{n}}
%EndExpansion
1/p_{j}$. Then%
\[
\frac{1}{p}%
%TCIMACRO{\dprod \limits_{j=1}^{n}}%
%BeginExpansion
{\displaystyle\prod\limits_{j=1}^{n}}
%EndExpansion
q_{j}^{p}\leq%
%TCIMACRO{\dsum \limits_{j=1}^{n}}%
%BeginExpansion
{\displaystyle\sum\limits_{j=1}^{n}}
%EndExpansion
\frac{1}{p_{j}}q_{j}^{p_{j}}%
\]
regardless of the choices of $q_{1},..,q_{n}\geq0.$
\end{lemma}

Now we are ready to prove the aforementioned theorem:

\begin{theorem}
\label{gpdt}A map $f\in{\mathcal{H}}$ is $R_{1},...,R_{t}$-$S$-abstract
$(p_{1},...,p_{t})$-summing if and only if there is a constant $C>0$ and Borel
probability measures $\mu_{j}$ on $K_{j}$ such that%
\begin{equation}
S(f,x^{(1)},...,x^{(r)},b^{(1)},...,b^{(t)})\leq C%
%TCIMACRO{\dprod \limits_{j=1}^{t}}%
%BeginExpansion
{\displaystyle\prod\limits_{j=1}^{t}}
%EndExpansion
\left(  \int_{K_{j}}R_{j}\left(  \varphi,x^{(1)},...,x^{(r)},b^{(j)}\right)
^{p_{j}}d\mu_{j}\right)  ^{1/p_{j}} \label{2}%
\end{equation}
for all $x^{(l)}\in E_{l},$ $l=1,...,r$ and $b^{(j)}\in G_{j}$, with
$j=1,...,t.$
\end{theorem}

\begin{proof}
One direction is canonical and we omit. Let us suppose that $f\in{\mathcal{H}%
}$ is $R_{1},...,R_{t}$-$S$-abstract $(p_{1},...,p_{t})$-summing. Consider the
compact sets $P(K_{k})$ of the probability measures in $C(K_{k})^{\ast}$, for
all $k=1,...,t$. For each $(x_{j}^{(l)})_{j=1}^{m}$ in $E_{l}$ and
$(b_{j}^{(s)})_{j=1}^{m}$ in $G_{s},$ with $(s,l)\in\{1,...,t\}\times
\{1,...,r\},$ let%
\[
g=g_{(x_{j}^{(l)})_{j=1}^{m},(b_{j}^{(s)})_{j=1}^{m},(s,l)\in\{1,...,t\}\times
\{1,...,r\}}:P(K_{1})\times\cdots\times P(K_{t})\rightarrow\mathbb{R}%
\]
be defined by%
\begin{align*}
&  g\left(  (\mu_{j})_{j=1}^{t}\right)  =\\
&  =\sum_{j=1}^{m}\left[  \frac{1}{p}S(f,x_{j}^{(1)},...,x_{j}^{(r)}%
,b_{j}^{(1)},...,b_{j}^{(t)})^{p}-C^{p}\sum_{k=1}^{t}\frac{1}{p_{k}}%
\int_{K_{k}}R_{k}\left(  \varphi,x_{j}^{(1)},...,x_{j}^{(r)},b_{j}%
^{(k)}\right)  ^{p_{k}}d\mu_{k}\right]  .
\end{align*}
As usual, the family $\mathcal{F}$ of all such $g$'s is concave and one can
also easily prove that every $g\in\mathcal{F}$ is convex and continuous.
Besides, for each $g\in\mathcal{F}$ there are measures $\mu_{j}^{g}\in
P(K_{j}),$ $j=1,...,t$, such that%
\[
g(\mu_{1}^{g},...,\mu_{t}^{g})\leq0.
\]
In fact, using the compactness of each $K_{k}$ ($k=1,...,t$), the continuity
of $\left(  R_{k}\right)  _{x_{j}^{(1)},...,x_{j}^{(r)},b_{j}^{(k)}},$ there
are $\varphi_{k}\in K_{k}$ so that%
\[
\sum_{j=1}^{m}R_{k}\left(  \varphi_{k},x_{j}^{(1)},...,x_{j}^{(r)},b_{j}%
^{(k)}\right)  ^{p_{k}}=\sup_{\varphi\in K_{k}}\sum_{j=1}^{m}R_{k}\left(
\varphi,x_{j}^{(1)},...,x_{j}^{(r)},b_{j}^{(k)}\right)  ^{p_{k}}.
\]
Now, with the Dirac measures $\mu_{k}^{g}=\delta_{\varphi_{k}},$ $k=1,...,t,$
and Lemma \ref{yy} we get%
\[
g(\mu_{1}^{g},...,\mu_{t}^{g})\leq0.
\]
So, Ky Fan's Lemma asserts that there are $\overline{\mu_{j}}\in P(K_{j}),$
$j=1,...,t,$ so that%
\[
g(\overline{\mu_{1}},...,\overline{\mu_{t}})\leq0
\]
for all $g\in\mathcal{F}$. Hence
\[
\sum_{j=1}^{m}\left[  \frac{1}{p}S(f,x_{j}^{(1)},...,x_{j}^{(r)},b_{j}%
^{(1)},...,b_{j}^{(t)})^{p}\right]  -C^{p}\sum_{k=1}^{t}\frac{1}{p_{k}}%
\int_{K_{k}}\sum_{j=1}^{m}R_{k}\left(  \varphi,x_{j}^{(1)},...,x_{j}%
^{(r)}b_{j}^{(k)}\right)  ^{p_{k}}d\overline{\mu_{k}}\leq0
\]
and from the particular case $m=1$ we obtain%
\begin{equation}
\frac{1}{p}S(f,x^{(1)},...,x^{(r)},b^{(1)},...,b^{(t)})^{p}\leq C^{p}%
\sum_{k=1}^{t}\frac{1}{p_{k}}\int_{K_{k}}R_{k}\left(  \varphi,x^{(1)}%
,...,x^{(r)},b^{(k)}\right)  ^{p_{k}}d\overline{\mu_{k}}. \label{new}%
\end{equation}
If $x^{(1)},...,x^{(r)},b^{(1)},...,b^{(t)}$ are given and, for $k=1,...,t,$
define
\[
\tau_{k}:=\left(  \int_{K_{k}}R_{k}\left(  \varphi,x^{(1)},...,x^{(r)}%
,b^{(k)}\right)  ^{p_{k}}d\overline{\mu_{k}}\right)  ^{1/p_{k}}.
\]
If $\tau_{k}=0$ for every $k$ then, the result is immediate. Let us now
suppose that $\tau_{j}$ is not zero for some $j\in\{1,...,t\}$. Consider
\[
V=\{j\in\{1,..,t\};\tau_{j}\neq0\}
\]
and $\beta>0$ big enough to get
\begin{equation}
0<\left(  \tau_{j}\beta^{\frac{1}{pp_{j}}}\right)  ^{-1}<1\text{ for every
}j\in V. \label{bbbvvv}%
\end{equation}
The above condition is necessary in view of (\ref{novaq}). Consider, also,%
\[
\vartheta_{j}=\left\{
\begin{array}
[c]{c}%
\left(  \tau_{j}\beta^{\frac{1}{pp_{j}}}\right)  ^{-1}\text{ if }j\in V\\
1\text{ if }j\notin V.
\end{array}
\right.
\]
Thus, since $0<\vartheta_{j}\leq1,$ we have
\begin{align*}
\frac{1}{p}S(f,x^{(1)},...,x^{(r)},\vartheta_{1}b^{(1)},...,\vartheta
_{t}b^{(t)})^{p}  &  \leq C^{p}\sum_{k=1}^{t}\frac{1}{p_{k}}\int_{K_{k}}%
R_{k}\left(  \varphi,x^{(1)},...,x^{(r)},\vartheta_{k}b^{(k)}\right)  ^{p_{k}%
}d\overline{\mu_{k}}\\
&  \leq C^{p}\sum_{k\in V}\frac{1}{p_{k}}\left(  \tau_{k}\beta^{\frac
{1}{pp_{k}}}\right)  ^{-p_{k}}\tau_{k}^{p_{k}}\\
&  \leq\frac{C^{p}}{p}\frac{1}{\beta^{\frac{1}{p}}}%
\end{align*}
and%
\begin{equation}
S(f,x^{(1)},...,x^{(r)},b^{(1)},...,b^{(t)})^{p}\leq C^{p}\beta^{\left(
%TCIMACRO{\tsum \nolimits_{j\in V}}%
%BeginExpansion
{\textstyle\sum\nolimits_{j\in V}}
%EndExpansion
1/p_{j}\right)  -1/p}%
%TCIMACRO{\tprod \nolimits_{j\in V}}%
%BeginExpansion
{\textstyle\prod\nolimits_{j\in V}}
%EndExpansion
\tau_{j}^{p}. \label{wwss}%
\end{equation}
If $V\neq\{1,...,t\}$, then
\[
\frac{1}{p}-%
%TCIMACRO{\dsum \nolimits_{j\in V}}%
%BeginExpansion
{\displaystyle\sum\nolimits_{j\in V}}
%EndExpansion
\frac{1}{p_{j}}>0.
\]
Note that it is possible to make $\beta\rightarrow\infty$ in (\ref{wwss}),
since it does not contradict (\ref{bbbvvv}); so we get%
\[
S(f,x^{(1)},...,x^{(r)},b^{(1)},...,b^{(t)})^{p}=0
\]
and we again reach (\ref{2}). The case $V=\{1,...,t\}$ is immediate.
\end{proof}

\subsection{Application: The (general) Unified PDT and the case of dominated
multilinear mappings}

By choosing $r=t=n=1$ in Theorem \ref{gpdt} we obtain an improvement of the
Unified Pietsch Domination Theorem from \cite{BPRn}. In fact, we obtain
precisely \cite[Theorem 2.1]{psjmaa} which is essentially the general unified
PDT (we just need to repeat the trick used in \cite[Theorem 3.1]{psjmaa}).

It is interesting to note that, in the case $n>1$, the trick used in
\cite[Theorem 3.1]{psjmaa} is essentially what emerges the notion of weighted
summability. In resume, this trick works perfectly for $n=1$, but for other
cases it forces us to deal with weighted summability. So, one shall not expect
for the possible relaxation of conditions (\ref{novaq}) for the validity of
Theorem 4.6.

As pointed out in the introduction, contrary to what happens in \cite{BPRn},
our theorem straightforwardly recovers the domination theorem for
$(q_{1},...,q_{n})$-dominated $n$-linear mappings (with $1/q=1/q_{1}%
+\cdots+1/q_{n}$). In fact, we just need to choose%

\[
\left\{
\begin{array}
[c]{c}%
t=n\\
G_{j}=X_{j}\text{ and }K_{j}=B_{X_{j}^{\ast}}\text{ for all }j=1,...,n\\
E_{j}=\mathbb{K},j=1,...,r\text{ }\\
\mathcal{H}=\mathcal{L}(X_{1},...,X_{n};Y)\\
p_{j}=q_{j}\text{ for all }j=1,...,n\\
S(T,x^{(1)},...,x^{(r)},b^{(1)},...,b^{(n)})=\left\Vert T(b^{(1)}%
,...,b^{(n)})\right\Vert \\
R_{k}(\varphi,x^{(1)},...,x^{(r)},b^{(k)})=\left\vert \varphi(b^{(k)}%
)\right\vert \text{ for all }k=1,...,n.
\end{array}
\right.
\]
So, with these choices, $T$ is $R_{1},..,R_{n}$-$S$ abstract $(q_{1}%
,...,q_{n})$-summing precisely when $T$ is $(q_{1},...,q_{n})$-dominated. In
this case Theorem \ref{gpdt} tells us that there is a constant $C>0$ and there
are measures $\mu_{k}$ on $K_{k}$, $k=1,...,n,$ so that%
\[
S(T,x^{(1)},...,x^{(r)},b^{(1)},...,b^{(n)})\leq C%
%TCIMACRO{\dprod \limits_{k=1}^{n}}%
%BeginExpansion
{\displaystyle\prod\limits_{k=1}^{n}}
%EndExpansion
\left(  \int_{K_{k}}R_{k}\left(  \varphi,x^{(1)},...,x^{(r)},b^{(k)}\right)
^{q_{k}}d\mu_{k}\right)  ^{\frac{1}{q_{k}}},
\]
i.e.,%
\[
\left\Vert T(b^{(1)},...,b^{(n)})\right\Vert \leq C%
%TCIMACRO{\dprod \limits_{k=1}^{n}}%
%BeginExpansion
{\displaystyle\prod\limits_{k=1}^{n}}
%EndExpansion
\left(  \int_{K_{k}}\left\vert \varphi(b^{(k)})\right\vert ^{q_{k}}d\mu
_{k}\right)  ^{\frac{1}{q_{k}}}.
\]

\subsection{Application: The PDT for Cohen strongly $q$-summing operators
\label{ko}}

The class of Cohen strongly $q$-summing multilinear operators was introduced
by D. Achour and L. Mezrag in \cite{AMe}. Let $1<q<\infty$ and $X_{1}%
,...,X_{n},Y$ arbitrary Banach spaces. If $q>1$, then $q^{\ast}$ denotes the
real number satisfying $1/q+1/q^{\ast}=1.$ A continuous $n$-linear operator
$T:X_{1}\times\cdots\times X_{n}\rightarrow Y$ is Cohen strongly $q$-summing
if and only if there is a constant $C>0$ such that for any positive integer
$m,$ $x_{1}^{(j)},...,x_{m}^{(j)}$ in $X_{j}$ ($j=1,...,n$) and any
$y_{1}^{\ast},...,y_{m}^{\ast}$ in $Y^{\ast}$, the following inequality hold:
\[%
%TCIMACRO{\dsum \limits_{i=1}^{m}}%
%BeginExpansion
{\displaystyle\sum\limits_{i=1}^{m}}
%EndExpansion
\left\vert y_{i}^{\ast}\left(  T(x_{i}^{(1)},...,x_{i}^{(n)})\right)
\right\vert \leq C\left(
%TCIMACRO{\dsum \limits_{i=1}^{m}}%
%BeginExpansion
{\displaystyle\sum\limits_{i=1}^{m}}
%EndExpansion%
%TCIMACRO{\dprod \limits_{j=1}^{n}}%
%BeginExpansion
{\displaystyle\prod\limits_{j=1}^{n}}
%EndExpansion
\left\Vert x_{i}^{(j)}\right\Vert ^{q}\right)  ^{1/q}\sup_{y^{\ast\ast}\in
B_{Y^{\ast\ast}}}\left(
%TCIMACRO{\dsum \limits_{i=1}^{m}}%
%BeginExpansion
{\displaystyle\sum\limits_{i=1}^{m}}
%EndExpansion
\left\vert y^{\ast\ast}(y_{i}^{\ast})\right\vert ^{q^{\ast}}\right)
^{1/q^{\ast}}.
\]
In the same paper the authors also prove the following Pietsch-type theorem:

\begin{theorem}
[Achour-Mezrag]A continuous $n$-linear mapping $T:X_{1}\times\cdots\times
X_{n}\rightarrow Y$ is Cohen strongly $q$-summing if and only if there is a
constant $C>0$ and a probability measure $\mu$ on $B_{Y^{\ast\ast}}$ so that
for all $(x^{(1)},...,x^{(n)},y^{\ast})$ in $X_{1}\times\cdots\times
X_{n}\times Y^{\ast}$ the inequality%
\begin{equation}
\left\vert y^{\ast}\left(  T(x^{(1)},...,x^{(n)})\right)  \right\vert \leq
C\left(
%TCIMACRO{\dprod \limits_{k=1}^{n}}%
%BeginExpansion
{\displaystyle\prod\limits_{k=1}^{n}}
%EndExpansion
\left\Vert x^{(k)}\right\Vert \right)  \left(  \int_{B_{Y^{\ast\ast}}%
}\left\vert y^{\ast\ast}(y^{\ast})\right\vert ^{q^{\ast}}d\mu\right)
^{\frac{1}{q^{\ast}}} \label{qqaa}%
\end{equation}
is valid.
\end{theorem}

Note that by choosing the parameters%
\[
\left\{
\begin{array}
[c]{c}%
t=2\text{ and }r=n\\
E_{i}=X_{i}\text{ for all }i=1,...,n\text{ }\\
K_{1}=B_{X_{1}^{\ast}\times\cdots\times X_{n}^{\ast}}\text{ and }%
K_{2}=B_{Y^{\ast\ast}}\\
G_{1}=\mathbb{K}\text{ and }G_{2}=Y^{\ast}\\
\mathcal{H}=\mathcal{L}(X_{1},...,X_{n};Y)\\
p=1,\text{ }p_{1}=q\text{ and }p_{2}=q^{\ast}\\
S(T,x^{(1)},...,x^{(n)}, b, y^{\ast})=\left\vert y^{\ast}\left(
T(x^{(1)},...,x^{(n)}\right)  \right\vert \\
R_{1}(\varphi,x^{(1)},...,x^{(n)},b)=\left\Vert x^{(1)}\right\Vert
\cdots\left\Vert x^{(n)}\right\Vert \text{ }\\
R_{2}(\varphi,x^{(1)},...,x^{(n)},y^{\ast})=\left\vert \varphi(y^{\ast
})\right\vert
\end{array}
\right.
\]
we can easily conclude that $T:X_{1}\times\cdots\times X_{n}\rightarrow Y$ is
Cohen strongly $q$-summing if and only if $T$ is $R_{1},R_{2}$-$S$ abstract
$(q,q^{\ast})$-summing. Theorem \ref{gpdt} tells us that $T$ is $R_{1},R_{2}%
$-$S$ abstract $(q,q^{\ast})$-summing if and only if there is a $C>0$ and
there are probability measures $\mu_{k}$ in $K_{k}$, $k=1,2$, such that
\begin{align*}
S(T,x^{(1)},...,x^{(n)},b,y^{\ast})  &  \leq C\left(  \int_{K_{1}}R_{1}\left(
\varphi,x^{(1)},...,x^{(n)},b\right)  ^{q}d\mu_{1}\right)  ^{\frac{1}{q}}\\
&  \left(  \int_{K_{2}}R_{2}\left(  \varphi,x^{(1)},...,x^{(n)},y^{\ast
}\right)  ^{q^{\ast}}d\mu_{2}\right)  ^{\frac{1}{q^{\ast}}},
\end{align*}
i.e.,%
\begin{align*}
\left\vert y^{\ast}\left(  T(x^{(1)},...,x^{(n)})\right)  \right\vert  &  \leq
C\left(  \int_{B_{X_{1}^{\ast}\times\cdots\times X_{n}^{\ast}}}\left(
\left\Vert x^{(1)}\right\Vert \cdots\left\Vert x^{(n)}\right\Vert \right)
^{q}\text{ }d\mu_{1}\right)  ^{\frac{1}{q}}\left(  \int_{B_{Y^{\ast\ast}}%
}\left\vert \varphi(y^{\ast})\right\vert ^{q^{\ast}}d\mu_{2}\right)
^{\frac{1}{q^{\ast}}}\\
&  =C\left\Vert x^{(1)}\right\Vert ...\left\Vert x^{(n)}\right\Vert \left(
\int_{B_{Y^{\ast\ast}}}\left\vert \varphi(y^{\ast})\right\vert ^{q^{\ast}}%
d\mu_{2}\right)  ^{\frac{1}{q^{\ast}}}%
\end{align*}
and we recover (\ref{qqaa}) regardless of the choice of the positive integer
$m$ and $x^{(k)}\in X_{k}$, $k=1,...,n$.

\section{Weighted summability}

The notion of weighted summability (see the comments just after Theorem
\ref{ttta}) emerged from the paper \cite{psmz} as a natural concept when we
were dealing with problem (\ref{domGG}).

In this section we observe that this concept in fact emerges in more abstract
situations and seems to be unavoidable in further developments of the
nonlinear theory.

Let $0<q_{1},...,q_{n}<\infty,$ $1/q=%
%TCIMACRO{\tsum \limits_{j=1}^{n}}%
%BeginExpansion
{\textstyle\sum\limits_{j=1}^{n}}
%EndExpansion
1/q_{j},$ $X_{1},...,X_{n}$ be Banach spaces and
\[
A:Map(X_{1},...,X_{n};Y)\times X_{1}\times\cdots\times X_{n}\rightarrow
\lbrack0,\infty)
\]
be an arbitrary map. Let us say that $f\in Map(X_{1},...,X_{n};Y)$ is
$A$-$(q_{1},...,q_{n})$-dominated if there is a constant $C>0$ so that%

\begin{equation}
A(f,x^{(1)},...,x^{(n)})\leq C\left(  {\displaystyle\int\nolimits_{B_{X_{1}%
^{\ast}}}}\left\vert \varphi(x^{(1)})\right\vert ^{q_{1}}d\mu_{1}\right)
^{\frac{1}{q_{1}}}\cdot\dots\cdot\left(  {\displaystyle\int\nolimits_{B_{X_{n}%
^{\ast}}}}\left\vert \varphi(x^{(n)})\right\vert ^{q_{n}}d\mu_{k}\right)
^{\frac{1}{q_{n}}}, \label{fAdomi}%
\end{equation}
regardless of the choice of the positive integer $m$ and $x^{(k)}\in X_{k}$,
$k=1,...,n$.

In fact, more abstract maps could be used in the right-hand side of
(\ref{fAdomi}). However, since our intention is illustrative rather than
exhaustive, we prefer to deal with this more simple case.

\begin{theorem}
\label{cinco}An arbitrary map $f\in Map(X_{1},...,X_{n};Y)$ is $A$%
-$(q_{1},...,q_{n})$-dominated if there exists $C>0$ such that%
\begin{equation}
\left(  \sum_{j=1}^{m}\left(  \left\vert b_{j}^{(1)}...b_{j}^{(n)}\right\vert
A(f,x_{j}^{(1)},...,x_{j}^{(n)})\right)  ^{q}\right)  ^{\frac{1}{q}}\leq C%
%TCIMACRO{\dprod \limits_{k=1}^{n}}%
%BeginExpansion
{\displaystyle\prod\limits_{k=1}^{n}}
%EndExpansion
\sup_{\varphi\in B_{X_{k}^{\ast}}}\left(
%TCIMACRO{\dsum \limits_{j=1}^{m}}%
%BeginExpansion
{\displaystyle\sum\limits_{j=1}^{m}}
%EndExpansion
\left(  \left\vert b_{j}^{(k)}\right\vert \left\vert \varphi(x_{j}%
^{(k)})\right\vert \right)  ^{q_{k}}\right)  ^{1/q_{k}} \label{adomi}%
\end{equation}
for every positive integer $m$, $(x_{j}^{(k)},b_{j}^{(k)})\in X_{k}%
\times\mathbb{K}$, with $(j,k)\in\{1,...,m\}\times\{1,...,n\}.$
\end{theorem}

\begin{proof}
Choosing the parameters%
\[
\left\{
\begin{array}
[c]{c}%
r=t=n\\
E_{j}=X_{j}\text{ and }G_{j}=\mathbb{K}\text{ for all }j=1,...,n\\
K_{j}=B_{X_{j}^{\ast}}\text{ for all }j=1,...,n\\
\mathcal{H}=Map(X_{1},...,X_{n};Y)\\
p=q\text{ and }p_{j}=q_{j}\text{ for all }j=1,...,n\\
S(f,x^{(1)},...,x^{(n)},b^{(1)},...,b^{(n)})=\left\vert b^{(1)}...b^{(n)}%
\right\vert A(f,x^{(1)},...,x^{(n)})\\
R_{k}(\varphi,x^{(1)},...,x^{(n)},b^{(k)})=\left\vert b^{(k)}\right\vert
\left\vert \varphi(x^{(k)})\right\vert \text{ for all }k=1,...,n.
\end{array}
\right.
\]
we easily conclude that $\left(  \ref{adomi}\right)  $ holds if and only if
$f$ is $R_{1},..,R_{n}$-$S$ abstract $(q_{1},...,q_{n})$-summing. In this case
Theorem \ref{gpdt} tells us that there is a constant $C>0$ and there are
measures $\mu_{k}$ on $K_{k}$, $k=1,...,n,$ such that%
\[
S(T,x^{(1)},...,x^{(n)},b^{(1)},...,b^{(n)})\leq C{\displaystyle\prod
\limits_{k=1}^{n}}\left(  \int_{K_{k}}R_{k}\left(  \varphi,x^{(1)}%
,...,x^{(n)},b^{(k)}\right)  ^{q_{k}}d\mu_{k}\right)  ^{\frac{1}{q_{k}}},
\]
i.e.,
\[
\left\vert b^{(1)}...b^{(n)}\right\vert A(f,x^{(1)},...,x^{(n)})\leq
C{\displaystyle\prod\limits_{k=1}^{n}}\left(  \int_{B_{X_{k}^{\ast}}}\left(
\left\vert b^{(k)}\right\vert \left\vert \varphi(x^{(k)})\right\vert \right)
^{q_{k}}d\mu_{k}\right)  ^{\frac{1}{q_{k}}},
\]
for all $(x^{(k)},b^{(k)})\in X_{k}\times\mathbb{K}$, $k=1,...,n$, and we
readily obtain $\left(  \ref{fAdomi}\right)  $.
\end{proof}

\begin{remark}
As we have mentioned before, the procedure of this last section is
illustrative. The interested reader can easily find a characterization similar
to Theorem \ref{cinco} in the full abstract context of Definition \ref{quatro}.
\end{remark}

\end{document}